\documentclass[12pt, amsfonts]{amsart}

\usepackage{cite}                       
\usepackage{amsfonts,amssymb}           

\newtheorem{thm}{Theorem~}
\newtheorem{cor}[thm]{Corollary~}
\newtheorem{lem}[thm]{Lemma~}

\theoremstyle{remark}

\newtheoremstyle{specthm}{1.5ex plus 1ex minus .2ex}{1.5ex plus 1ex minus .2ex}{\it}{}{\bf}{}{1em}{\thmnote{#3}}
\theoremstyle{specthm}

\def\R{{\mathbb R}}

\def\e{\varepsilon}

\newcommand{\abs}[1]{\left\vert{#1}\right\vert}         
\newcommand{\set}[1]{\left\{{#1}\right\}}               
\def\vol{\mbox{\rm Vol}}
\newcommand{\sr}{S^{n-1}}                               
\newcommand{\normk}[2]{\left\Vert{#2}\right\Vert_{#1}}  
\newcommand{\rdn}[2]{\mathcal{R}_{#1}\left(#2\right)}   
\newcommand{\grsm}[1]{G(n,#1)}                          
\newcommand{\intk}[2]{\int_0^{#1}t^{n-k-1}f\left(#2\right)d t}

\begin{document}


\title[Stability and slicing inequalities]
{Stability and slicing inequalities for intersection bodies}

\author{Alexander Koldobsky}

\address{Department of Mathematics\\
University of Missouri\\
Columbia, MO 65211}

\email{koldobskiya@missouri.edu}

\author{Dan Ma}

\address{Department of Mathematics\\
University of Missouri\\
Columbia, MO 65211}

\email{madan516@gmail.com}



\begin{abstract} We prove a generalization of the hyperplane inequality
for intersection bodies, where volume is replaced by an arbitrary measure $\mu$ with
even continuous density and sections are of arbitrary dimension $n-k,\ 1\le k <n.$
If $K$ is a generalized $k$-intersection body, then
$$\mu(K)\,\leq\,\frac{n}{n-k}c_{n,k}\max_{H} \mu(K\cap H) \vol_n(K)^{k/n}.$$
Here $c_{n,k} = |B_2^n|^{(n-k)/n}/|B_2^{n-k}|<1,$
$|B_2^n|$ is the volume of the unit Euclidean ball, and maximum is taken over all
$(n-k)$-dimensional subspaces of $\R^n.$ The constant is optimal, and
for each intersection body the inequality holds for every $k.$ We also prove a stronger
``difference" inequality. The proof is based on stability in the lower dimensional
Busemann-Petty problem for arbitrary measures in the following sense. Let $\e>0,\ 1\le k <n.$ Suppose
that $K$ and $L$ are origin-symmetric star bodies in $\R^n,$ and
$K$ is a generalized $k$-intersection body. If for every $(n-k)$-dimensional subspace
$H$ of $\R^n$
$$\mu(K\cap H)\leq \mu(L\cap H)+\e,$$
then
$$\mu(K)\leq \mu(L) +\frac{n}{n-k}c_{n,k} \vol_n(K)^{k/n}\e.$$
\end{abstract}
\maketitle


\section{Introduction}

The Busemann-Petty problem, posed in 1956 in \cite{BP56}, asks the following question.
Suppose that $K$ and $L$ are origin-symmetric convex bodies in $\R^n$ so that
\[\vol_{n-1}(K\cap \xi^\bot)\leq\vol_{n-1}(L\cap \xi^\bot),\quad\forall \xi\in\sr{},\]
where $\xi^\bot$ is the central hyperplane perpendicular to $\xi.$
Does it follow that
\[\vol_n(K)\leq\vol_n(L)?\]
The answer is affirmative if $n\le 4$ and negative if $n\ge 5.$
The solution was completed at the end of the 90's as the result of
a sequence of papers \cite{LR}, \cite{Ba}, \cite{Gi}, \cite{Bo4},
\cite{L}, \cite{Pa}, \cite{G1}, \cite{G2}, \cite{Z1}, \cite{Z2}, \cite{K1}, \cite{K2}, \cite{Z3},
\cite{GKS}; see \cite[p. 3]{K3} or \cite[p. 343]{G3} for details.

It is natural to ask what happens if hyperplane sections are replaced by sections of lower dimensions.
Suppose that for every ($n-k$)-dimensional subspace $H\in\R^n$,
\[\vol_{n-k}(K\cap H)\leq\vol_{n-k}(L\cap H).\]
Does it follow that
\[\vol_n(K)\leq\vol_n(L)?\]
Zhang \cite{Zha96} proved that the answer is affirmative if and only if
all origin-symmetric convex bodies in $\R^n$ are generalized $k$-intersection
bodies (see definition in Section 2; this is similar to the connection between the original
Busemann-Petty problem and intersection bodies established by Lutwak in \cite{L}).
Using this connection, Bourgain and Zhang \cite{BZ99} proved  that
the answer is negative if the dimension of sections $n-k > 3$ (see also \cite{RZ04} and different
later proof  in \cite{Kol00}).
However, the cases of two- and three-dimensional sections remain open. Other results on
the lower dimensional Busemann-Petty problem can be found
in  \cite{Mil06, Mil07, Mil08, Yas06a, Yas06b, Rub09, RZ04}.

In this paper, we establish stability in the affirmative part of the lower dimensional Busemann-Petty problem.
Stability problems in convex geometry have been considered for a long time; see \cite{Gro96} for numerous results
and references. Stability in volume comparison problems was first studied in \cite{Kola}, where such results were
proved for the Busemann-Petty and Shephard problems. We extend the result of \cite[Theorem 1]{Kola} to sections
of lower dimensions in the following way.

\begin{thm}\label{thm:stbp}
    Let $K$ and $L$ be origin-symmetric star bodies in $\R^n$,
    and $1\le k<n$. Suppose $K$ is a generalized $k$-intersection body and $\e>0.$
    If for every $(n-k)$-dimensional subspace $H$ of $\R^n$
    \begin{equation}\label{eqn:sect}
        \vol_{n-k}(K\cap H)\leq\vol_{n-k}(L\cap H)+\e,
    \end{equation}
    then
    \begin{equation}\label{eqn:comp}
        \vol_n(K)^{\frac{n-k}n}\leq\vol_n(L)^{\frac{n-k}n}+c_{n,k}\ \e,
    \end{equation}
    where $c_{n,k} = |B_2^n|^{(n-k)/n}/|B_2^{n-k}|$ and $|B_2^n|$ is the volume
of the unit Euclidean ball.
\end{thm}
Note that $c_{n,k}<1,$ which immediately follows from the log-convexity of the $\Gamma$-function
(see for example \cite[Lemma 2.1]{KL}). Also, in the formulation of Theorem 1 in \cite{Kola} the constant
$c_{n,1}$ was replaced by 1, though the proof there gives the result with $c_{n,1}.$

\smallbreak
Zvavitch \cite{Zv} found a remarkable generalization of the Busemann-Petty problem
to arbitrary measures. It appears that one can replace volume by any measure with
even continuous density in $\R^n.$ Let $f$ be an even continuous non-negative
function on $\R^n,$ and denote by $\mu$ the measure on $\R^n$ with density $f$.
For every closed bounded set $B\subset \R^n$ define
$$\mu(B)=\int\limits_B f(x)\ dx.$$ It was proved in \cite{Zv} that, for $n\le 4$  and any
origin-symmetric convex bodies $K$ and $L$ in $\R^n,$ the inequalities
$$\mu(K\cap \xi^\bot) \le \mu(L\cap \xi^\bot), \qquad \forall \xi\in S^{n-1}$$
imply
$$\mu(K)\le \mu(L). $$
Zvavitch also proved that this is generally not true if $n\ge 5,$ namely, for any $\mu$ with strictly positive
even continuous density there exist $K$ and $L$ providing a counterexample.

Stability in Zvavitch's result was established in \cite[Theorem 2]{K6}. Here we extend this result
to sections of lower dimensions, as follows.
\begin{thm}\label{thm:stbpgm}
    Let $K$ and $L$ be origin-symmetric star bodies in $\R^n$, and $1<k<n$.
    Suppose $K$ is a generalized $k$-intersection body and $\e>0.$
    If for every $(n-k)$-dimensional subspace $H$ of $\R^n$
    \begin{equation}\label{eqn:mucomp}
        \mu(K\cap H)\leq\mu(L\cap H)+\varepsilon,
    \end{equation}
    then
    \[\mu(K)\leq\mu(L)+{\frac{n}{n-k}}c_{n,k}\vol_n(K)^{k/n}\varepsilon.\]
\end{thm}
In the case $f\equiv 1,$ we get another stability result for volume which is weaker than what is provided
by Theorem \ref{thm:stbp}. This is the reason why we state Theorem \ref{thm:stbp} separately. However,
for arbitrary measures the constant in Theorem \ref{thm:stbpgm} is the best possible, as follows from
the example after Corollary \ref{hyper-meas}.

The stability results mentioned above were applied in \cite{Kola,K6} to the
hyperplane (or slicing) problem of Bourgain \cite{Bo1,Bo2} that can be formulated as follows.
Does there exist an absolute constant $C$ so that for any origin-symmetric convex body $K$ in $\R^n$
\begin{equation} \label{hyper}
\vol_n(K)^{\frac {n-1}n} \le C \max_{\xi \in S^{n-1}} \vol_{n-1}(K\cap \xi^\bot) ?
\end{equation}
The best-to-date estimate $C\sim n^{1/4}$ is due to Klartag \cite{Kl},
who removed the logarithmic term from the previous estimate of Bourgain \cite{Bo3}.
We refer the reader to recent papers \cite{DP,EK} for the history and 
current state of the hyperplane problem.

In the case where $K$ is an intersection body (see Section 2 for definitions and properties),
the inequality (\ref{hyper}) is known for sections of arbitrary dimension
with the best possible constant. For any $1\le k<n,$
\begin{equation}\label{hyper-inter}
\vol_n(K)^{\frac {n-k}n} \le c_{n,k}
\max_{H\in G(n,n-k)} \vol_{n-k}(K\cap H),
\end{equation}
where $G(n,n-k)$ is the Grassmanian of $(n-k)$-dimensional subspaces of $\R^n,$
and the equality is attained when $K=B_2^n.$ In particular, if the dimension $n\le 4,$
then (\ref{hyper-inter}) is true for any origin-symmetric convex body $K.$ The proof
is an immediate consequence of Zhang's connection between generalized intersection
bodies and the lower dimensional Busemann-Petty problem; apply this connection to any
generalized $k$-intersection body $K$ and $L=B_2^n.$ Then use the fact that every intersection
body is a generalized $k$-intersection body for every $k$ (see \cite{GZ} or \cite{Mil06}).
For every fixed $k,$ the inequality (\ref{hyper-inter}) holds for any generalized $k$-intersection body $K.$

We prove several generalizations of (\ref{hyper-inter}) using the stability results formulated above.
First, interchanging $K$ and $L$ in Theorem \ref{thm:stbp}, we get the following ``difference" inequality,
previously established in \cite[Corollary 1]{Kola} in the hyperplane case.
\begin{cor}\label{thm:stbpa}
    Let $K$ and $L$ be origin-symmetric star bodies in $\R^n$,
    and $1\le k<n$. Suppose $K$ and $L$ are  generalized $k$-intersection bodies, then
    $$\abs{\vol_n(K)^{\frac{n-k}n}-\vol_n(L)^{\frac{n-k}n}}$$
    $$\leq c_{n,k} \max_{H\in G(n,n-k)}\left|\vol_{n-k}(K\cap H)-\vol_{n-k}(L\cap H)\right|.$$
\end{cor}
Putting $L=\varnothing$ in the latter inequality, we get (\ref{hyper-inter}) for any generalized
$k$-intersection body $K.$

Interchanging $K$ and $L$ in Theorem \ref{thm:stbpgm},
we get the following inequality, which was earlier proved for $k=1$ in \cite[Corollary 1]{K6}.
\begin{cor} \label{ineq-meas} Let $K$ and $L$ be origin-symmetric star bodies in $\R^n$, and $1\le k<n$.
    Suppose that $K$ and $L$ are generalized $k$-intersection bodies. Then
$$\left|\mu(K) - \mu(L)\right| \le $$
$$
\frac {n}{n-k}c_{n,k} \max_{H}\left|\mu(K\cap H) - \mu(L\cap H)\right|
 \max\left\{\vol_n(K)^{k/n}, \vol_n(L)^{k/n}\right\},
$$
where maximum is taken over all $(n-k)$-dimensional subspaces $H$ of $\R^n.$
\end{cor}
Putting $L=\varnothing,$ we generalize to lower dimensions the hyperplane inequality
for arbitrary measures from \cite[Theorem 1]{K6}.
\begin{cor}\label{hyper-meas} Let $1\le k <n,$ and suppose that $K$ is a generalized $k$-intersection
body in $\R^n.$ Then
\begin{equation} \label{arbmeas}
\mu(K) \le \frac n{n-k} c_{n,k} \max_{H\in G(n,n-k)} \mu(K\cap H)\ \vol_n(K)^{k/n}.
\end{equation}
\end{cor}

The constant in the right-hand side is the best possible. In fact, let $K=B_2^n$
and, for every $j\in N,$  let $f_j$ be a non-negative continuous function on $[0,1]$ supported in $(1-\frac 1j,1)$
and such that $\int_0^1 f_j(t) dt =1.$ Let $\mu_j$ be the measure on $\R^n$ with density $f_j(|x|_2),$
where $|x|_2$ is the Euclidean norm.  We have
$$\mu_j(B_2^n) = |S^{n-1}| \int_0^1 r^{n-1}f_j(r) dr,$$
where $|S^{n-1}|= 2\pi^{n/2}/\Gamma(n/2)$ is the surface area of the unit sphere in $\R^n.$
For every $H\in G(n,n-k)$,
$$\mu_j(B_2^n \cap H) = |S^{n-k-1}| \int_0^1 r^{n-k-1} f_j(r) dr.$$
Clearly,
$$\lim_{j\to \infty} \frac{\int_0^1 r^{n-1}f_j(r) dr}{\int_0^1 r^{n-k-1}f_j(r) dr} =1.$$
Using $|S^{n-1}|=n|B_2^n|,$ we get
$$\lim_{j\to \infty} \frac{\mu_j(B_2^n)}{\max_{H} \mu_j(B_2^n\cap H)\ \vol_n(B_2^n)^{k/n}}
= \frac{|S^{n-1}|}{|S^{n-k-1}||B_2^n|^{k/n}} = \frac n{n-k}c_{n,k},$$
which shows that the constant is asymptotically optimal.


\section{Stability}
We say that a closed bounded set $K$ in $\R^n$ is a {\it star body}  if
every straight line passing through the origin crosses the boundary of $K$
at exactly two points different from the origin, the origin is an interior point of $K,$
and the {\it Minkowski functional}
of $K$ defined by
$$\|x\|_K = \min\{a\ge 0:\ x\in aK\}$$
is a continuous function on $\R^n.$

The {\it radial function} of a star body $K$ is defined by
$$\rho_K(x) = \|x\|_K^{-1}, \qquad x\in \R^n.$$
If $x\in S^{n-1}$ then $\rho_K(x)$ is the radius of $K$ in the
direction of $x.$

Writing the volume of $K$ in polar coordinates, one gets
\begin{equation} \label{polar-volume}
\vol_n(K)
=\frac{1}{n} \int_{S^{n-1}} \rho_K^n(\theta) d\theta=
\frac{1}{n} \int_{S^{n-1}} \|\theta\|_K^{-n} d\theta.
\end{equation}

The {\it spherical Radon transform}
$R:C(S^{n-1})\mapsto C(S^{n-1})$
is a linear operator defined by
$$Rf(\xi)=\int_{S^{n-1}\cap \xi^\bot} f(x)\ dx,\quad \xi\in S^{n-1}$$
for every function $f\in C(S^{n-1}).$

The polar formula (\ref{polar-volume}) for the volume of a hyperplane section expresses this volume
in terms of the spherical Radon transform (see for example \cite[p.15]{K3}):
\begin{equation} \label{volume=spherradon}
S_K(\xi)=\vol_{n-1}(K\cap \xi^\bot) =
\frac{1}{n-1} R(\|\cdot\|_K^{-n+1})(\xi).
\end{equation}

The spherical Radon
transform is self-dual (see \cite[Lemma 1.3.3]{Gro96}):
for any functions $f,g\in C(S^{n-1})$
\begin{equation} \label{selfdual}
\int_{S^{n-1}} Rf(\xi)\ g(\xi)\ d\xi = \int_{S^{n-1}} f(\xi)\ Rg(\xi)\ d\xi.
\end{equation}

Using self-duality, one can extend the spherical Radon transform to measures.
Let $\mu$ be a finite Borel measure on $S^{n-1}.$
We define the spherical Radon transform of $\mu$ as a functional $R\mu$ on
the space $C(S^{n-1})$ acting by
$$(R\mu,f)= (\mu, Rf)= \int_{S^{n-1}} Rf(x) d\mu(x).$$
By Riesz's characterization of continuous linear functionals on the
space $C(S^{n-1})$,
$R\mu$ is also a finite Borel measure on $S^{n-1}.$ If $\mu$ has
continuous density $g,$ then by (\ref{selfdual}) the
Radon transform of $\mu$ has density $Rg.$

The class of intersection bodies was introduced by Lutwak \cite{L}.
Let $K, L$ be origin-symmetric star bodies in $\R^n.$ We say that $K$ is the
intersection body of $L$ if the radius of $K$ in every direction is
equal to the $(n-1)$-dimensional volume of the section of $L$ by the central
hyperplane orthogonal to this direction, i.e. for every $\xi\in S^{n-1},$
\begin{equation} \label{intbodyofstar}
\rho_K(\xi)= \|\xi\|_K^{-1} = \vol_{n-1}(L\cap \xi^\bot).
\end{equation} \index{intersection body of a star body}
All the bodies $K$ that appear as intersection bodies of different star bodies
form {\it the class of intersection bodies of star bodies}.

Note that the right-hand
side of (\ref{intbodyofstar}) can be written in terms of the spherical Radon transform using (\ref{volume=spherradon}):
$$\|\xi\|_K^{-1}= \frac{1}{n-1} \int_{S^{n-1}\cap \xi^\bot} \|\theta\|_L^{-n+1} d\theta=
\frac{1}{n-1} R(\|\cdot\|_L^{-n+1})(\xi).$$
It means that a star body $K$ is
the intersection body of a star body if and only if the function $\|\cdot\|_K^{-1}$
is the spherical Radon transform of a continuous positive function on $S^{n-1}.$
This allows to introduce a more general class of bodies. A star body
$K$ in $\R^n$ is called an {\it intersection body}
if there exists a finite Borel measure \index{intersection body}
$\mu$ on the sphere $S^{n-1}$ so that $\|\cdot\|_K^{-1}= R\mu$ as functionals on
$C(S^{n-1}),$ i.e. for every continuous function $f$ on $S^{n-1},$
\begin{equation} \label{defintbody}
\int_{S^{n-1}} \|x\|_K^{-1} f(x)\ dx = \int_{S^{n-1}} Rf(x)\ d\mu(x).
\end{equation}

Intersection bodies played the crucial role in the solution of the original Busemann-Petty problem
due to the following connection found by Lutwak \cite{L}. If $K$ in an origin-symmetric
intersection body in $\R^n$ and $L$ is any origin-symmetric star body in $\R^n,$
then the inequalities $S_K(\xi)\le S_L(\xi)$ for all $\xi\in S^{n-1}$ imply that
$\vol_n(K)\le \vol_n(L),$ i.e. the answer to the Busemann-Petty problem in this
situation is affirmative. For more information about intersection bodies, see \cite[Chapter 4]{K3},
\cite{KY}, \cite[Chapter 8]{G3} and references there. In particular, every origin-symmetric convex
body in $\R^n,\ n\le 4$ is an intersection body; see \cite{G2, Z3, GKS}. Also the unit ball of any
finite dimensional subspace of $L_p,\ 0<p\le 2$ is an intersection body; see \cite{K1}.

Zhang in \cite{Zha96} introduced a generalization of intersection bodies.
For $1\leq k\leq n-1$, the \emph{($n-k$)-dimensional spherical Radon transform} is an operator
$\mathcal{R}_{n-k}:C(\sr{})\mapsto C(\grsm{n-k})$ defined by
\[\rdn{n-k}{f}(H)=\int_{\sr{}\cap H}f(x)d x,\quad H\in\grsm{n-k}.\]
Denote the image of the operator $\mathcal{R}_{n-k}$ by X:
$$\rdn{n-k}{C(\sr{})}=X\subset C(\grsm{n-k}).$$
Let $M^+(X)$ be the space of linear positive continuous functionals on $X$, i.e. for every
$\nu\in M^+(X)$ and non-negative function $f\in X$, we have $\nu(f)\geq0$.

An origin-symmetric star body $K$ in $\R^n$ is called a \emph{generalized $k$-intersection
body} if there exists a functional $\nu\in M^+(X)$, so that for every $f\in C(\sr{})$,
\[\int_{\sr{}}\normk{K}{x}^{-k}f(x)d x=\nu(\rdn{n-k}{f}).\]

When $k=1$ we get the class of intersection bodies.
It was proved by Grinberg and Zhang \cite[Lemma 6.1]{GZ} that every intersection body in $\R^n$
is a generalized $k$-intersection
body for every $k<n.$ More generally, as proved later by Milman \cite{Mil06}, if  $m$ divides $k$, then every
generalized $m$-intersection body is a generalized $k$-intersection body.  Zhang \cite{Zha96}
showed that the answer to the lower dimensional
Busemann-Petty problem is affirmative if and only if every origin-symmetric convex body
in $\R^n$ is a generalized $k$-intersection body.

Denote by $1_S\equiv 1$ and $1_G\equiv 1$ the functions which are equal to 1
everywhere on the unit sphere $S^{n-1}$ and the Grassmanian $G(n,n-k),$ correspondingly.
Then, $\rdn{n-k}{1_S}=|S^{n-k-1}|\ 1_G.$
\smallbreak
We are now ready to prove the stability in the lower dimensional Busemann-Petty problem.
\begin{proof}[Proof of Theorem \ref{thm:stbp}]
By the polar formula for volume (\ref{polar-volume}),  for each $H\in\grsm{n-k}$ we have
    \begin{equation}\label{eqn:rdn}
        \vol_{n-k}(K\cap H)=\frac{1}{n-k}\rdn{n-k}{\normk{K}{\cdot}^{-n+k}}(H),
    \end{equation}
    Then the inequality (\ref{eqn:sect}) can be written as
    \begin{equation}\label{eqn:sectrd}
        \rdn{n-k}{\normk{K}{\cdot}^{-n+k}}(H)\leq \rdn{n-k}{\normk{L}{\cdot}^{-n+k}}(H)+(n-k)\varepsilon.
    \end{equation}
    Since $K$ is a generalized $k$-intersection body, there exists $\mu_0\in M^+$, such that for each $\psi\in C\left(\sr\right)$,
    \begin{equation}\label{eqn:ms}
        \int_{\sr{}}\normk{K}{x}^{-k}\psi(x)d x=\,\mu_0(\rdn{n-k}{\psi}).
    \end{equation}
    Since $\mu_0$ is a positive functional, by (\ref{eqn:sectrd}) and (\ref{eqn:ms}), we have
    \begin{eqnarray}
        n\vol_n(K) &=& \int_{\sr}\normk{K}{x}^{-k}\normk{K}{x}^{-n+k}d x\notag\\
        &=& \mu_0\left(\rdn{n-k}{\normk{K}{\cdot}^{-n+k}}\right)\notag\\
        &\le& \mu_0\left(\rdn{n-k}{\normk{L}{\cdot}^{-n+k}}\right)+(n-k)\varepsilon \mu_0(1_G)\notag\\
        \label{eqn:sp}&:=& \mathrm{I}+\mathrm{II}.
    \end{eqnarray}
    Using (\ref{eqn:ms}), H\"{o}lder's inequality and polar formula for the volume, we get
    \begin{eqnarray}
        \mathrm{I} &=& \int_{\sr}\normk{K}{x}^{-k}\normk{L}{x}^{-n+k}d x\notag\\
        &\leq& \left(\int_{\sr}\normk{K}{x}^{-n}d x\right)^{k/n}\left(\int_{\sr}\normk{L}{x}^{-n}d x\right)^{(n-k)/n}\notag\\
        \label{eqn:sp1}&=& n\vol_n(K)^{k/n}\vol_n(L)^{(n-k)/n}.
    \end{eqnarray}
   Now, by (\ref{eqn:ms}), the well-known formula $|S^{n-1}|=n|B_2^n|$  (see \cite[p. 33]{K3}) and H\"older's inequality,
    \begin{eqnarray*}
        \mathrm{II} &=& (n-k)\e \mu_0(1_G)= \frac{(n-k)\e}{|S^{n-k-1}|} \int_{S^{n-1}} \|x\|_K^{-k}\ 1_S(x)\ dx\\
        &\le&  \frac{(n-k)\e}{|S^{n-k-1}|} \left(\int_{S^{n-1}} \|x\|_K^{-n}dx\right)^{k/n}|S^{n-1}|^{\frac{n-k}n}\\
        &=&\frac{n^{k/n}(n-k)|S^{n-1}|^{\frac{n-k}n}}{|S^{n-k-1}|} \vol_n(K)^{k/n} \e\\
        &=& \frac{n |B_2^n|^{\frac{n-k}n}}{|B_2^{n-k-1}|}  \vol_n(K)^{k/n} \e.\\
   \end{eqnarray*}
  Combining this with (\ref{eqn:sp}) and (\ref{eqn:sp1}), we get the result.
   \end{proof}
\medbreak

We now pass to stability for arbitrary measures. Let $\mu$ be a measure on $\R^n$ with even continuous
density $f.$ Let $\chi$ be the indicator function of the interval $[0,\,1]$. The measure $\mu$ of a star body $K$ can
be expressed in polar coordinates as follows:
\begin{eqnarray}
    \mu(K) &=& \int_K f(x)\ dx = \int_{\R^n} \chi(\|x\|_K)f(x)\ dx\notag\\
    \label{meas-polar}&=& \int_{S^{n-1}} \left(\int_0^{\|\theta\|_K^{-1}} r^{n-1}f(r\theta)\ dr \right) d\theta.
\end{eqnarray}

Similarly,  we can express the volume of a section of $K$ by an $(n-k)$ -dimensional subspace $H$ of $\R^n$ as
\begin{eqnarray}
    \mu(K\cap H) &=& \int_H\chi\left(\normk{K}{x}\right)f(x)d x\notag\\
    &=& \int_{\sr\cap H}\left(\intk{\normk{K}{\theta}^{-1}}{t\theta}\right)d\theta\notag\\
    \label{sect-polar} &=&  \rdn{n-k}{\int_0^{\|\theta\|_K^{-1}} r^{n-k-1} f(r\theta)\ dr }(H),
\end{eqnarray}
where the Radon transform is applied to a function of the variable $\theta\in S^{n-1}.$

We need the following lemma, which was also used by Zvavitch in his proof.

\begin{lem}\label{thm:elem}
    Let $a,b,k\in\R^+$, and $\alpha$ be a non-negative function on $(0,\max\set{a,b})$,
    such that the integral below converges. Then
    \begin{equation*}\label{lemma-zv}
        \int_0^ar^{n-1}\alpha(r)\ dt-a^k\int_0^ar^{n-k-1}\alpha(r)\ dr
    \end{equation*}
      $$ \leq\int_0^br^{n-1}\alpha(r)\ dr - a^k\int_0^br^{n-k-1}\alpha(r)\ dr$$
\end{lem}
\begin{proof}
    The result follows from
    \[a^k\int_a^br^{n-k-1}\alpha(r)\ dr\leq\int_a^br^{n-1}\alpha(r)\ dr.\]
\end{proof}

\begin{proof}[Proof of Theorem \ref{thm:stbpgm}]
    Using (\ref{sect-polar}), inequality (\ref{eqn:mucomp}) can be written as
   \begin{equation} \label{eqn:murcomp}
  \rdn{n-k}{\int_0^{\|\theta\|_K^{-1}} r^{n-k-1} f(r\theta)\ dr }(H)
    \end{equation}
   $$\le \rdn{n-k}{\int_0^{\|\theta\|_L^{-1}} r^{n-k-1} f(r\theta)\ dr }(H) + \e, \qquad \forall H\in G(n,n-k).$$

 As in the proof of Theorem \ref{thm:stbp}, let $\mu_0$ be the positive functional
    associated with the generalized $k$-intersection body $K.$
    Applying $\mu_0$ to both sides of  (\ref{eqn:murcomp})
    and then using (\ref{eqn:ms}), we get
    \begin{equation}\label{eqn:mukcomp}
   \int_{S^{n-1}} \|\theta\|_K^{-k} \left(\int_0^{\|\theta\|_K^{-1}} r^{n-k-1} f(r\theta)\ dr\right) d\theta
    \end{equation}
    $$\le \int_{S^{n-1}} \|\theta\|_K^{-k} \left(\int_0^{\|\theta\|_L^{-1}} r^{n-k-1} f(r\theta)\ dr\right) d\theta
    + \e \mu_0(1_G).$$
    Applying Lemma \ref{thm:elem} with $a=\normk{K}{\theta}^{-1}$, $b=\normk{L}{\theta}^{-1}$
    and $\alpha(r)=f(r\theta)$ and then integrating over the sphere, we get
   $$ \int_0^{\|\theta\|_K^{-1}} r^{n-1} f(r\theta)\ dr - \|\theta\|_K^{-k} \int_0^{\|\theta\|_K^{-1}} r^{n-k-1} f(r\theta)\ dr$$
   $$\le  \int_0^{\|\theta\|_L^{-1}} r^{n-1} f(r\theta)\ dr - \|\theta\|_K^{-k} \int_0^{\|\theta\|_L^{-1}} r^{n-k-1} f(r\theta)\ dr,$$

    and

    \begin{eqnarray}
    \label{eqn:elek}&&\int_{S^{n-1}} \left(\int_0^{\|\theta\|_K^{-1}} r^{n-1} f(r\theta)\ dr \right) d\theta\\ 
   &-& \int_{S^{n-1}} \|\theta\|_K^{-k} \left(\int_0^{\|\theta\|_K^{-1}} r^{n-k-1} f(r\theta)\ dr \right) d\theta\notag\\
   &\le&  \int_{S^{n-1}} \left(\int_0^{\|\theta\|_L^{-1}} r^{n-1} f(r\theta)\ dr \right) d\theta \notag\\
   &-& \int_{S^{n-1}} \|\theta\|_K^{-k} \left(\int_0^{\|\theta\|_L^{-1}} r^{n-k-1} f(r\theta)\ dr \right) d\theta.\notag
   \end{eqnarray}
    Adding (\ref{eqn:mukcomp}) and (\ref{eqn:elek}) and using (\ref{meas-polar}) we get
    $$\mu(K)\leq\mu(L)+\varepsilon \mu_0(1_G).$$
    As shown in the proof of Theorem \ref{thm:stbp},
    $$\mu_0(1_G)\leq\frac{n}{n-k}c_{n,k}\vol_n(K)^{k/n},$$
    which completes the proof.
\end{proof}
As mentioned earlier, every intersection body is a generalized $k$-intersection body for every $k,$
so if $K$ is an intersection body, the results of Theorems \ref{thm:stbp} and \ref{thm:stbpgm} hold
for all $k$ at the same time, as well as the results of Corollaries \ref{thm:stbpa}, \ref{ineq-meas}, \ref{hyper-meas}.
\bigbreak

{\bf Acknowledgement.} The first named author wishes to thank
the US National Science Foundation for support through
grant DMS-1001234.


\end{document}